\documentclass[a4paper,12pt]{article}
\usepackage{amsmath, amssymb}
\usepackage[english]{babel}
  \usepackage{paralist}
  \usepackage{graphics} 
  \usepackage{epsfig} 
\usepackage{graphicx}
\usepackage{epstopdf}
 \usepackage[colorlinks=true]{hyperref}
 \usepackage{array}
\usepackage[latin1]{inputenc}
\usepackage{xcolor}
\usepackage{esint}
\usepackage{latexsym,amsfonts}
\usepackage{amsthm}
\usepackage{verbatim}
\usepackage[bb=boondox]{mathalfa}

\usepackage{array}
\newtheorem{theorem}{Theorem}[section]
\newtheorem{cor}{Corollary}[section]

\newtheorem{lem}{Lemma}[section]

\theoremstyle{definition}

\newtheorem*{theorem*}{Theorem}

\numberwithin{equation}{section}

\newcommand{\D}{\mathbb D}
\newcommand{\T}{\partial \mathbb{D}}
\newcommand{\bigint}{\begin{picture}(10,10)
\put(-1,2){\line(1,0){10}}
\end{picture}\kern-14pt\int}

\def\XXint#1#2#3{{\setbox0=\hbox{$#1{#2#3}{\int}$}
     \vcenter{\hbox{$#2#3$}}\kern-.5\wd0}}

\title{Contractive analytic self-mappings of the disc }

\author{Artur Nicolau}

\date{}

\begin{document}
 
\maketitle

\begin{abstract}
Analytic self-maps of the unit disc whose hyperbolic derivative is uniformly bounded by a constant smaller than one, are called contractive. We describe these maps in terms of their Aleksandrov-Clark measures and in terms of their inner-outer factorization. In addition, we show that contractive inner functions can be described in terms of a certain mixing property of its boundary values. We also present other results on the boundary behavior of contractive inner functions. 
\end{abstract}

\footnote{The author is supported in part by the Generalitat de Catalunya (grant 2021 SGR 00071), the Spanish Ministerio de Ciencia e Innovaci\'on (project PID2021-123151NB-I00) and the Spanish Research Agency through the Mar\'ia de Maeztu Program (CEX2020-001084-M).}

\section{Introduction}

Let $\D$ be the unit disc in the complex plane and let $\T$ be the unit circle. Let $d_h (z, w)$ be the hyperbolic distance between the points $z,w \in \D$, defined by 
$$
d_h (z,w) = \frac{1}{2} \log \frac{1+ \rho(z,w)^2}{1- \rho (z,w)^2}, \qquad \rho (z,w) = \bigg| \frac{z-w}{1- \overline{w}z}\bigg|.
$$
The Schwarz Lemma says that any analytic mapping $f: \D \rightarrow \D$ contracts hyperbolic distances, that is, $d_h (f(z), f(w)) \leq d_h (z,w)$ for any $z,w \in \D$, or equivalently its derivative $D_h (f)$ with respect the hyperbolic metric is bounded by $1$. A calculation shows that
$$
D_h (f) (z) = \lim_{w \to z } \frac{d_h (f(w), f(z))}{d_h (w,z)} =  \frac{(1-|z|^2) |f'(z)|}{1- |f(z)|^2}, \quad z \in \D .
$$
Moreover if there exists a point $z_0 \in \D$ such that $D_h (f) (z_0) = 1$, then $D_h (f) (z) =1$ for any $z \in \D$ and $f$ is an automorphism of $\D$. Hence, the hyperbolic derivative $D_h (f) (z)$ measures how much $f$ deviates from being an automorphism of $\D$ near the point $z$. 

An analytic mapping $f: \D \rightarrow \D$ is called contractive if there exists a constant $D<1$ such that $d_h (f(z), f(w)) \leq D \,  d_h (z,w)$ for any $z,w \in \D$, or equivalently, if $D(f)=\sup \{D_h (f) (z) : z \in \D \} < 1$. One nice feature of contractive analytic mappings is conformal invariance, that is, $\tau_1 \circ f \circ
\tau_2$ is contractive whenever $f$ is contractive and $\tau_1 , \tau_2$ are automorphisms of $\D$. Moreover $D(\tau_1 \circ f \circ \tau_2) = D(f)$. The main goal of this paper is to describe analytic self-mappings of $\D$ in different terms: their Aleksandrov-Clark measures; their inner-outer factorization and their boundary behavior. 

\medskip

Given an analytic mapping $f: \D \rightarrow \D$ and a value $\alpha \in \partial \D$, the function $(\alpha + f)/(\alpha - f)$ has positive real part. Hence there exist a positive measure $\sigma_\alpha = \sigma_\alpha (f)$ on $\partial \D$ and a constant $C_\alpha \in \mathbb{R}$ such that
\begin{equation}
    \label{AC}
    \frac{\alpha + f(z)}{\alpha - f(z)} = \int_{\partial \D} \frac{\xi + z}{\xi - z} d \sigma_\alpha (\xi) + i C_\alpha , \quad z \in \D,  
\end{equation}
or equivalently,   
\begin{equation}
    \label{AC11}
    \frac{1-|f(z)|^2}{|\alpha - f(z)|^2} = \int_{\partial \D} \frac{1-|z|^2}{|\xi - z|^2} d \sigma_\alpha (\xi), \quad z \in \D .
\end{equation}
The measures $\{\sigma_\alpha : \alpha \in \partial \D\}$ are called the Aleksandrov-Clark measures of $f$. They were introduced by Clark in the seventies and after the work of Aleksandrov in the nineties, have become an important topic in function theory. Actually many properties of an analytic selfmapping $f$ of $\D$ can be translated to properties of the Aleksandrov-Clark measures of $f. $ The nice surveys \cite{PolSar} and \cite{Saks}, as well as chapter 9 of \cite{cima}, contain many of their properties and a wide range of applications. We mention that $f$ is an inner function if and only if $\sigma_\alpha$ is a singular measure with respect to Lebesgue measure $dm$ on $\partial \D$ for some (and then all) $\alpha \in \partial \D$. Moreover in this case, $\sigma_\alpha$ is concentrated on the set $f^{-1} (\{\alpha \})$. 

Let $dA$ denote the normalized Lebesgue measure on $\D$ and for $1<p<\infty$ let $L^p (dA)$ be the standard Lebesgue spaces. The Bergman projection $P$ given by
$$
P (g) (z) = \int_{\D} \frac{g(w)}{(1- \overline{w}z)^2} d A(w), \quad z\in \D, \quad g \in L^2(dA), 
$$
is the orthogonal projection from $L^2(dA)$ onto the closed subspace of analytic functions in $\D$ which are in $L^2 (\D)$. It is well known that $P$ is a bounded operator from $L^p (dA)$ into $L^p (dA)$ for any $1<p<\infty$. Given an arc $I \subset \partial \D$ consider the Carleson box $Q= Q(I) =  \{z \in \D : 1-|z| \leq m(I), z/|z| \in I \}$. We will also use the standard notation $l(Q) = m(I)= |I|$ if $Q=Q(I)$. Given a positive integrable function $w$ in $\D$ consider the weighted space $L^p (w dA)$ of measurable functions $g$ in $\D$ such that $ \int_{\D} |g(z)|^p w(z) dA(z) < \infty $. Bekoll\'e and Bonami studied the boundedness of the Bergman projection on weighted Lebesgue spaces and proved that $P$ is a bounded operator from $L^p (w dA)$ to $L^p (w dA)$ if and only the weight $w$ satisfies
\begin{equation}
    \label{BB}
    \sup \bigg(\frac{1}{A(Q)} \int_Q w dA \bigg) \bigg( \frac{1}{A(Q)} \int_Q w^{-1/(p-1)} dA \bigg)^{p-1} < \infty , 
\end{equation}
where the supremum is taken over all Carleson boxes $Q \subset \D$. See \cite{BB}. The weights $w$ satisfying \eqref{BB} are called Bekoll\'e-Bonami $B_p$ weights.

Our first result is a description of contractive mappings in terms of their Aleksandrov-Clark measures, in which Bekoll\'e-Bonami $B_2$ weights play a central role. 

\begin{theorem}
\label{B2weights}
    Let $f: \mathbb{D} \rightarrow \mathbb{D}$ be an analytic mapping. Let $\sigma$ be an Aleksandror-Clark measure of $f$ and let $u$ be its Poisson integral defined by 
    $$
u (z) = \int_{\partial \D} \frac{1-|z|^2}{|\xi - z|^2} d \sigma (\xi), \quad z \in \D. 
$$
The following conditions are equivalent:

$(a)$ The mapping $f$ is contractive.

$(b)$ There exists a constant $c>0$ such that
\begin{equation}
\label{boxest}
   \frac{\sigma(I)}{|I|} \geqslant c u \left(z_{I}\right) ,  
\end{equation}
for any arc $ I \subset \partial \mathbb{D}$. Here $z_{I}=(1-|I|) \xi_{I}$, where $\xi_{I}$ denotes the center of $I$.

$(c)$ $u$ is a Bekoll\'e-Bonami $B_2$ weight, that is,
$$
[u]_{B_2} : =\sup \left(\frac{1}{A(Q)} \int_{Q} u d A\right)\left( \frac{1}{A(Q)} \int_{Q} \frac{dA}{u} \right)<\infty, 
$$
where the supremum is taken over all Carleson boxes $Q \subset \mathbb{D}$.

\end{theorem}

We now mention some remarks. 

\textbf{1. Symmetric measures.} Recall that a bounded analytic function $f$ in $\D$ is called inner if it has radial limits of modulus $1$ at almost every point of the unit circle. Let $f$ be an inner function. In \cite{AAN} it was shown that $D_h (f) (z) \rightarrow 0$ as $|z| \rightarrow 1$ if and only if $\sigma$ is symmetric, that is,
$$
\lim _{|I| \rightarrow 0} \frac{\sigma(I)}{\sigma\left(I^{\prime}\right)}=1 .
$$
Here $I'$ is an arc contiguous to $I$ of the same length. This result can be understood as an asymptotic version of Theorem \ref{B2weights}.

\textbf{2. Box kernels.} The converse estimate to \eqref{boxest} in (b) holds for any positive measure $\sigma$. Indeed $\sigma (I) / |I| \leq 3 u (z_{I} )$, for any arc $I \subset \partial \mathbb{D}$ and any positive measure $\sigma$ on $\partial \mathbb{D}$. This follows from the fact that 
the Poisson kernel $P_z (\xi)= (1-|z|^2) |\xi - z|^{-2}$, $\xi \in \partial \D$, satisfies $P_z (\xi) \geq 3^{-1} (1-|z|)^{-1} {\mathbb{1}}_{I(z)} (\xi)$, where $I(z) = \{\xi \in \partial \D : |\xi - z/|z|| < 1-|z|\}$. Hence, condition \eqref{boxest} in (b) is equivalent to saying that $\sigma (I) / |I|$ is comparable to $u(z_I)$.
Equivalently, integrating $\sigma$ against the Poisson kernel $P_z $ is (uniformly in $z \in \D$) comparable to integrating against the box kernel, defined as $(1-|z|)^{-1} {\mathbb{1}}_{I(z)}$.   

\textbf{3. Doubling measures.} A positive measure $\sigma$ in $\partial \D$ is called a doubling measure if there exists a constant $C=C(\sigma) >0$ such that $\sigma (2I) \leq C \sigma (I)$ for any arc $I \subset \partial \D$. Here $2I$ refers to the arc on $\partial \D$ that has the same center as $I$ and length $ 2|I|$. The infimum of the constants $C>0$ satisfying the previous condition is called the doubling constant of $\sigma$. Condition (b) of Theorem \ref{B2weights} implies that $\sigma$ is a doubling measure but there are doubling measures which do not satisfy (b). However, if $\sigma$ is a doubling measure with doubling constant sufficiently close to $2$, then $\sigma$ satisfies condition (b). These facts will be discussed in Section \ref{section2} after the proof of Theorem \ref{B2weights}. 

\textbf{4. $B_p$ weights. } It will also be shown in Section \ref{section2} that any positive harmonic function is a $B_{p}$ weight for any $p>2$. On the other hand, there are positive harmonic functions $u$ in $\mathbb{D}$ such that $1 / u $ is not integrable in $\mathbb{D}$ and therefore cannot be  $B_{2}$ weights.

\textbf{5. Uniformity.} Let $\{\sigma_\alpha : \alpha \in \partial \D\}$ be the family of Aleksandrov-Clark measures of an analytic self-mapping of $\D$. As a consequence of Theorem \ref{B2weights}, if for some value $\alpha_0 \in \partial \D$ the measure $\sigma_{\alpha_0}$ verifies the condition in part (b)  of Theorem \ref{B2weights}, then $\sigma_\alpha$ verifies condition (b) for any $\alpha \in \partial \D$. Moreover, there exists a constant $C= C(\alpha_0) >0$ such that $[u_\alpha]_{B_2} \leq C [u_{\alpha_0}]_{B_2}$ for any $\alpha \in \partial \D$, where $u_\alpha$ denotes the Poisson integral of $\sigma_\alpha$. The analogous fact for Muckenhoupt weights was proved by Sarason in \cite{Sar} (see also \cite{AAD}).     

\vspace{0.2cm}

Recall that any bounded analytic function  $f$ in $\D$ factors as $F=EBS$ where $E$ is an outer function, $B$ is a Blaschke product and $S$ is a singular inner function associated to a singular measure on $\partial \D$. See Chapter II of \cite{garnett}. Our second description of contractive analytic mappings is given in terms of the inner-outer factorization, and it is inspired by the characterization of Bishop of bounded analytic functions in the little Bloch space (\cite{Bi}). Although there is no direct connection between Aleksandrov-Clark measures and the inner-outer factorization, the conditions in Theorems \ref{B2weights} and \ref{Zeros} are closely related. 

\begin{theorem}
    \label{Zeros}
    Let $f: \mathbb{D} \rightarrow \mathbb{D}$ be an analytic mapping. Let $\nu$ be the singular measure associated to its singular inner factor. Consider the measure
    $$
    \mu = \sum_{z \in \mathbb{D} : f(z)=0} (1-|z|^2) \delta_{z} + 2 \nu + 2 \log |f|^{-1} dm  
    $$
and its Poisson integral 
$$
P[\mu] (z) = \int_{\overline{\D}} \frac{1-|z|^2}{|1 - \overline{w} z |^2} d \mu (w), \quad z \in \D .
$$
The following conditions are equivalent: 

(a) $f$ is contractive.

(b) There exists a constant $C >0$ such that
$$
\frac{\mu (\overline{Q})}{l (Q)} \geq C P[\mu] (z_Q) \quad \text{ if } \quad P[\mu] (z_Q) \leq  C. 
$$
Here $z_Q = (1- l(Q)) \xi_Q$, where $\xi_Q$ is the center of the arc $\overline{Q} \cap \partial \D$. 
\end{theorem}

\vspace{0.2cm}

We now focus on contractive inner functions. It turns out that contractive inner functions can be described in terms of their mapping properties in $\partial \D$. Given a point $z \in \D$ and a measurable set $E \subset \partial \D$ let $\omega (z, E)$ denote the harmonic measure in $\D$ from the point $z$ of the set $E$, given by
\begin{equation}
    \label{harmonicmeasure}
    \omega(z, E)=\int_{E} \frac{1-|z|^{2}}{|\xi-z|^{2}} d m(\xi), \quad z \in \mathbb{D}, \, E \subset \partial \D .
\end{equation}

\begin{theorem}
\label{mixing}
Let $f$ be an inner function. The following conditions are equivalent:

(a) $f$ is contractive, that is,  $D = \sup \{D_h (f) (z) : z \in \D\} < 1$. 

(b) There exists a constant $C>0$ such that
\begin{equation}
    \label{0theoremmixing}
    C^{-1} \omega \left(f\left(z_{I}\right), E\right) \leqslant \frac{m\left(f^{-1}(E) \cap I\right)}{m(I)} \leqslant C \omega \left(f\left(z_{I}\right), E\right), 
\end{equation}
for any measurable set $E \subset \partial \mathbb{D}$ and any arc $I \subset \partial \mathbb{D}$. 
    
\end{theorem}

We make the following remarks. 

\textbf{1. The mixing property.} The second inequality in \eqref{0theoremmixing} is always fulfilled, that is, there exists a universal constant $C>0$ such that for any inner function $f$, any arc $I \subset \partial \D$ and any measurable set $E \subset \partial \D$, we have $
m (f^{-1}  (E)  \cap I ) \leq C \omega (f (z_{I}), E) m (I)$. This follows from the classical Loewner Lemma which says that $\omega (f(z), E) = \omega (z, f^{-1} (E))$ for any $z \in \D$, any measurable set $E \subset \partial \D$ and any inner function $f$ (see page 12 of \cite{Ahl}), and the estimate $P_{z_I} (\xi) \geq 3^{-1} {\mathbb{1}}_{I} (\xi) / m(I)$, $\xi \in \partial \D$. So, Theorem \ref{mixing} says that the first inequality in \eqref{0theoremmixing} holds if and only if $f$ is contractive. 

\textbf{2. Dependence of the constants.} Assume that (a) holds and let $C=C(D)$ be the constant appearing in part (b). The proof of Theorem \ref{mixing} shows that the constant $C=C(D)$ satisfies $ C(D) \longrightarrow 1$ as $D \rightarrow 0 $. The proof also provides local versions of the result. 

\vspace{0.2cm}

A measurable set $E \subset \partial \D$ is called a $B_2$ set if $\omega(z, E)$, as defined in \eqref{harmonicmeasure}, is a $B_2$ weight. By Theorem \ref{B2weights}, this is equivalent to the existence of a constant $C>0$ such that 
$$
\frac{m(I \cap E)}{m(I)} \geq C \omega (z_I, E),
$$
for any arc $I \subset \partial \D$. Note that by Remark 3 after Theorem \ref{B2weights}, if $E$ is a $B_2$ set,  the restriction of the Lebesgue measure to $E$ is a doubling measure and hence $B_2$ sets of positive measure must be well spread out over $\partial \D$. It turns out that there exist $B_2$ sets $E\subset \partial \D$ with $0<m(E) < 1$ and actually, contractive inner functions provide non-trivial examples of $B_2$ sets. It is convenient to mention that by Lowner's Lemma,  for any inner function $f$ we have $0< m(f^{-1} (E)) < 1$ for any measurable set $E \subset \partial \D$ with $0< m(E) < 1$. 

\begin{cor}\label{cor1}
Let $f$ be a contractive inner function. Then $f^{-1}(E)$ is a $B_2$ set for any measurable set $E \subset \partial \mathbb{D}$. 
\end{cor}

A related result was proved in \cite{AAN}.  More concretely, assume $D_h (f) (z)$ tends to $0$ as $|z| \to 1$. Then for any measurable set $E \subset \partial \mathbb{D}$, the restriction of Lebesgue measure to $f^{-1} (E)$ is a symmetric measure, that is,  
$$
\frac{m\left(f^{-1}(E) \cap I\right)}{m\left(f^{-1}(E) \cap I^{\prime}\right)} \xrightarrow[|I| \rightarrow 0]{ } 1 .
$$
Here $I'$ denotes an arc contiguous to $I$ of the same length. 

\vspace{0.2cm}

Our three main results are now applied to obtain fine properties of the boundary behavior of contractive inner functions. For $0<s<1$ let $M^{s}(E)$ denote the $s$-dimensional Hausdorff content of a set $E \subset \partial \mathbb{D}$ and let $ \operatorname{dim} (E)$ denote its Hausdorff dimension. Fern\'andez and Pestana proved that for any $0<s<1$ and for any inner function $f$ there exists a constant $C=C(s,f(0))$ such that $M^s (f^{-1} (E)) \geq C M^s (E)$ for any measurable set $E \subset \partial \D$. Therefore $\operatorname{dim} (f^{-1} (E) \geq \operatorname{dim} (E)$, for any measurable set $E \subset \partial \D$ and any inner function $f$. See \cite{FP}. Contractive inner functions have the following extreme behavior.    

\begin{theorem}
\label{preimage}
Let $f$ be a contractive inner function with
$
D=\sup \{ D_h (f) (z) : z \in  \mathbb{D}\} <1 .
$
Then $M^{1-D} ( f^{-1} ( \{\alpha \})  \cap I)  >0$ for any point  $\alpha \in \partial \mathbb{D}$ and any arc $I \subset \partial \mathbb{D}$.
\end{theorem}

Note that an inner function $f$ which extends analytically across an arc $I \subset \partial \D$ satisfies $D_h (f) (z)$ tends to $1$, as $z \in \D$ approaches any point of $I$. Hence, contractive inner functions cannot be analytically extended across any arc of the unit circle. Consequently, the set of singularities of a contractive inner function $f$ is the whole unit circle. Next, we will study the size of the set 
$$
E(f)= \{ \xi \in \partial \mathbb{D} : \limsup_{r \rightarrow 1}|f(r \xi)|<1 \}. 
$$
Since $f$ is inner, $m(E(f))=0$. However, we will show that if $f$ is a contractive inner function, then $\operatorname{dim}(E(f)) >0$. 


 \begin{theorem}
 \label{limsup}
Let $f$ be an inner function such that
$$
\sup \{D_h (f) (z) : z  \in \mathbb{D} \}<D<1 .
$$
Then there exists $0<t=t(D)<1$ such that
$$
\operatorname{dim}\left\{\xi \in I: \limsup_{r \rightarrow 1}|f(r\xi)| < t \right\}>1-D,
$$
for any arc $I \subset \partial \mathbb{D}$. In addition, $t(D) \rightarrow 0$ as $D \rightarrow 0$.
 \end{theorem}

Finally, we present some natural examples of contractive analytic selfmappings. A domain $\Omega \subset  \mathbb{D}$ will be called a contractive domain if any analytic mapping $f: \D \rightarrow \Omega$ satisfies $\sup \{D_h (f) (z) : z \in \D\} < 1$. We will show that a domain $\Omega \subset \D$ is contractive if and only if there exists a constant $R=R(\Omega) >0$ such that $\Omega$ does not contain any hyperbolic disc of hyperbolic radius $R$. This is analogous to the well known description of Bloch domains from \cite{ACP}. Contractive analytic mappings also appear naturally when studying composition operators in the classical Bloch space. We will also give other natural examples of inner functions and discuss when they are contractive and of bounded compression, which were considered in \cite{ivriinicolau2025}. 

This paper is organized as follows. Theorem \ref{B2weights} and the corresponding remarks are discussed  in Section \ref{section2}. Section \ref{section3} is devoted to the proofs of Theorems \ref{Zeros} and \ref{mixing}. Theorems \ref{preimage} and \ref{limsup} are proved in Section \ref{section4}. Finally  Section \ref{section5} contains the description of contractive domains, composition operators with contractive symbols and other natural examples. 

We use the standard notation $A \gtrsim B$ (respectively, $A \lesssim B$) if there exists an absolute constant $C>0$ such that $A \geq CB$ (respectively, $A \leq CB$). We also use the notation $ A \simeq B$ if $A\gtrsim B$ and $A \lesssim B$.

It is a pleasure to thank Alex Bergman for pointing my attention to \cite{Sar} and \cite{AAD} and to Oleg Ivrii for many interesting discussions.

\section{Proof of Theorem \ref{B2weights}}\label{section2}

Given an arc $I \subset \partial \D$ centered at $\xi_I$, recall the notations $z_I = (1- |I|)\xi_I$ and $Q(I)= \{z \in \D: 1-|z| \leq |I|, \, z/|z| \in I \}$. We also denote by $T(Q) = \{z \in Q: 1- |z| \geq l(Q) /2 \}$ the top part of $Q$. Harnack's inequality provides a universal constant $C>0$ such that $C^{-1} u(w) \leq u(z) \leq C u(w)$ for any pair of points $z,w \in T(Q)$, any Carleson box $Q$ and any positive harmonic function $u$ in $\D$. We start with an elementary auxiliary result. 

\begin{lem}
    \label{meanvalue}
There exists a universal constant $C>0$ such that 
$$
C^{-1} u \left(z_{I}\right) \leqslant \frac{1}{A(Q(I))} \int_{Q(I)} u d A \leqslant C u\left(z_{I}\right), 
$$
for any arc $I \subset \partial \mathbb{D}$ and any positive harmonic function $u$ in $\mathbb{D}$. 
    
\end{lem}

\begin{proof}
 The lower bound follows 
from the mean value property and Harnack's inequality. 
For the upper bound, we use that
$$
\frac{1}{|I|} \lesssim  \frac{|I|}{\left|\xi-z_{I}\right|^{2}} , \quad  \xi \in I ,
$$
which gives
$$
\frac{1}{|I|} \int_{I} u(r \xi) dm (\xi) \lesssim \int_{I} u( r \xi) \frac{|I|}{\left|\xi-z_{I}\right|^{2}} d m(\xi) \leq u\left(r z_{I}\right), \quad  0<r<1 .
$$
Applying Harnack's inequality one obtains
$$
\frac{1}{|I|^{2}} \int_{Q(I)} u d A \lesssim \frac{1}{|I|} \int_{1-|I|}^{1} u\left(r z_{I}\right) d r \lesssim u\left(z_{I}\right), 
$$
as desired.
\end{proof}

Let $f:\D \rightarrow \D$ be an analytic mapping and let $u$ be the Poisson integral of an Aleksandrov-Clark measure of $f$. We will use the following identities which follow from definition \eqref{AC} by a straightforward computation,
\begin{equation}
    \label{tag1}
    D_h (f) (z) = \frac{\left(1-|z|^{2}\right)\left|f^{\prime}(z)\right|}{1-|f(z)|^{2}}=\frac{\left(1-|z|^2 \right) |\nabla u(z)|}{2 u(z)}, \quad z \in \mathbb{D} ,
\end{equation}
and
\begin{equation}
    \label{tag2}
    \left(1-|z|^{2}\right)|\nabla u(z)|=2 \left|\int_{\partial \mathbb{D}} \frac{1-|z|^{2}}{|\xi-z|^{2} \tau_z (\xi)}  d \sigma (\xi) \right|, 
\end{equation}
where $\tau_z(\xi) = (\xi -z)/(1-\bar{z} \xi)$, $\xi \in \partial \D$. These identities have recently been used in \cite{APHA} and \cite{ivriinicolau2025}.

\medskip

We are now ready to prove Theorem \ref{B2weights}.

\begin{proof}[Proof of Theorem \ref{B2weights}.]

$(a) \Rightarrow(b)$. Fix $\varepsilon >0$. Assume $I \subset \partial  \mathbb{D}$ is an arc with 
$$
\frac{\sigma(I)}{ |I|} \leq \varepsilon u \left(z_{I}\right) . 
$$ 
We want to show that $\varepsilon$ is bounded below if $f$ is contractive. Let $0<\delta=\delta(\varepsilon)<1$ be a constant to be fixed later and let $\delta I $ denote the arc that has the same center as $I$ and length $\delta |I|$. By \eqref{tag2}, 
\begin{equation}
    \label{tag3}
\left(1-\left|z_{\delta I}\right|^{2}\right)\left|\nabla u\left(z_{\delta I}\right)\right|
 =2 \left| \int_{\partial \D} \frac{1-\left|z_{\delta I}\right|^{2}}{\left|\xi -z_{\delta I}\right|^{2} \tau_{z_{\delta I}} (\xi)} d \sigma (\xi) \right|.\end{equation}
Now the elementary estimate $|\xi - z_{\delta I}| \geq 1- |z_{\delta I}| = \delta |I|$ and Harnack's inequality give
\begin{equation}
    \label{tag311}
    \int_{I} \frac{1-\left|z_{\delta I}\right|^{2}}{\left|\xi-z_{\delta I} \right|^{2}} d \sigma (\xi) \leqslant \frac{ 2 \sigma(I)}{\delta|I|}
 \leqslant \frac{ 2 \varepsilon}{\delta} u \left(z_{I}\right) \leqslant \frac{ 2 \varepsilon}{\delta^{2}} u \left(z_{\delta I}\right). 
\end{equation}
Then
\begin{align}
    \label{estrella}
    & \left|\int_{\partial \D } \frac{1-|z_{\delta I}|^{2}}{\left|\xi-z_{\delta I}\right|^{2} \tau_{z_{\delta I}} (\xi)} d \sigma (\xi)\right|  \geq \\ \notag & \geq\left|\int_{\partial \D \setminus I} \frac{1-|z_{\delta I}|^{2}}{\left|\xi-z_{\delta I}\right|^{2} \tau_{z_{\delta I}} (\xi)} d \sigma (\xi)\right|  - \frac{2 \varepsilon}{\delta^2} u(z_{\delta I}). 
\end{align}
Since the argument of $\{\tau_{z_{\delta I}} (\xi) : \xi \in \partial \mathbb{D} \setminus I \}$ remains essentially constant (up to an error bounded by a fixed multiple of $\delta$), there exists an absolute constant $C>0$ such that 
$$
\left|\int_{\partial \D \setminus I} \frac{1-|z_{\delta I}|^{2}}{\left|\xi-z_{\delta I}\right|^{2} \tau_{z_{\delta I}} (\xi)} d \sigma (\xi)\right|  
 \geqslant(1- C\delta) \int_{\partial \D \setminus I} \frac{1- |z_{\delta I}|^{2}}{\left|\xi -z_{\delta I}\right|^{2}} d \sigma(\xi) . 
$$
Note that \eqref{tag311} gives that
$$
\int_{\partial \D \setminus I} \frac{1- |z_{\delta I}|^{2}}{\left|\xi -z_{\delta I}\right|^{2}} d \sigma(\xi) = u\left(z_{\delta I}\right)-\int_{I} \frac{1-\left|z_{\delta I}\right|^{2}}{\left|\xi -z_{\delta I}\right|^{2}} d \sigma (\xi)   \geqslant u\left(z_{\delta I}\right)\left( 1 -\frac{2 \varepsilon}{\delta^{2}} \right) . 
$$
Then using \eqref{estrella}, we obtain  
$$
\left|\int_{\partial \D } \frac{1-|z_{\delta I}|^{2}}{\left|\xi-z_{\delta I}\right|^{2} \tau_{z_{\delta I}} (\xi)} d \sigma (\xi)\right|  \geq   u (z_{\delta I})\bigg( (1 - C \delta ) (1- \frac{2 \varepsilon}{\delta^2}) - \frac{2 \varepsilon}{\delta^2} \bigg). 
$$
Hence, identity \eqref{tag3} gives 
$$
\left(1-\left|z_{\delta I}\right|^{2}\right)\left|\nabla u \left(z_{\delta I}\right)\right| \geqslant 2  u\left(z_{\delta I}\right) \bigg( (1 - C \delta ) (1- \frac{2 \varepsilon}{\delta^2}) - \frac{2 \varepsilon}{\delta^2} \bigg). 
$$
Now taking $\delta = \varepsilon^{1/4}$, the identity \eqref{tag1} shows that if $f$ is contractive, then $\varepsilon$ must be bounded below. 

\medskip

$(b) \Rightarrow(a)$. The idea is that condition (b) provides a certain cancelation in the integral appearing in \eqref{tag2}. More concretely, we will show that condition (b) implies that there exists a constant $0<D<1$ such that
\begin{equation}
    \label{AC1}
    \left|\int_{\partial \D} \frac{1-|z|^{2}}{|\xi-z|^{2} \tau_z (\xi)} d \sigma (\xi)\right| \leq D \int_{\partial \D} \frac{1-|z|^{2}}{|\xi-z|^{2}} d \sigma (\xi), \quad z \in \D  . 
\end{equation}
Split the unit circle into $12$ arcs of equal length $\partial \mathbb{D}=\bigcup_{i=1}^{12} J_{i}$ such that $J_{i}$ and $J_{i+6}$ are antipodal, $i=1, \ldots, 6$. Fix $z \in \mathbb{D}$ and consider the automorphism
$$
\tau (w)=\frac{z-w}{1-\bar{z} w}, \quad  w \in \overline{ \mathbb{D} } . 
$$
Let $I_{i}=\tau^{-1}\left(J_{i}\right), i=1, \ldots, 12$. Since the oscillation of the argument of $\{ \tau(\xi)$ : $\xi \in I_i \}$ is less than $\pi / 3$, we have that 
\begin{equation}
    \label{tag4}
    \left|\int_{I_{i}} \frac{1-|z|^{2}}{|\xi-z|^{2}} \frac{d \sigma(\xi)}{\tau(\xi)}\right| \simeq  \int_{I_{i}} \frac{1-|z|^{2}}{|\xi-z|^{2}} d \sigma ( \xi),\quad  i=1,\ldots , 12. 
\end{equation}
Note that each arc $I_{i} = I(z_i)$ for some $z_{i} \in \D$ with $d_{h}\left(z, z_{i}\right) \lesssim 1$, $i=1, \ldots 12$. Here, $I(z_i)$ denotes the arc centered at $z_i / |z_i|$ of length $1-|z_i|$. According to Remark 2, condition (b) is equivalent to $\sigma (I(z)) / |I(z)| \simeq u(z)$ for any $z \in \D $. Then, condition (b) and Harnack's inequality give 
$$
\int_{I_i} \frac{1-|z|^2}{|\xi - z|^2} d \sigma (\xi) \simeq \frac{\sigma\left(I (z_i)\right)}{| I (z_i)|} \simeq u(z_i) \simeq u(z), \quad i=1,\ldots, 12 .  
$$
Hence \eqref{tag4} gives 
$$
\left|\int_{I_i} \frac{1-|z|^{2}}{|\xi-z|^{2}} \frac{d \sigma(\xi)}{\tau(\xi)}\right| \simeq u (z), \quad i=1, \ldots , 12 .
$$
On the other hand, since $J_{i}$ and $J_{i+6}$ are antipodal, there exists a constant $0<C<2$ such that
$$
\sup \left\{ \left|\tau\left(\xi_{i}\right) + \tau\left(\xi_{i+6}\right)\right|: \xi_{i} \in I_{i}, \xi_{i+6} \in I_{i+6}\right\}<C ,\quad  i=1,\ldots , 6. 
$$
Hence, the complex numbers 
$$
\int_{I_{i}} \frac{1-|z|^{2}}{|\xi-z|^{2}} \frac{d \sigma(\xi)}{\tau(\xi)} \quad  \text { and } \quad  \int_{I_{i+6}} \frac{1-|z|^{2}}{|\xi-z|^{2}} \frac{d\sigma (\xi)}{\tau(\xi)}
$$
have comparable modulus and are far from being collinear, $i=1, \ldots , 6$. Hence there exists a constant $0<D<1$ such that
$$
\begin{aligned}
    & \left|\int_{I_{i}} \frac{1-|z|^{2}}{|\xi-z|^{2}} \frac{d \sigma(\xi)}{\tau(\xi)}+\int_{I_{i+6}} \frac{1-|z|^{2}}{|\xi-z|^{2}} \frac{d \sigma(\xi)}{\tau(\xi)}\right| \leq \\
& \leq D \left(\int_{I_{i}} \frac{1-|z|^{2}}{|\xi-z|^{2}} d \sigma (\xi)+\int_{I_{i+6}} \frac{1-|z|^{2}}{|\xi-z|^{2}} d \sigma (\xi)\right), \quad  i=1, \ldots, 6 .
\end{aligned}
$$
Summing over $i=1,\ldots, 6$, we obtain \eqref{AC1}, which according to the identities \eqref{tag1} and \eqref{tag2},  finishes the proof.

\medskip 

$(a) \Rightarrow(c)$. By identity \eqref{tag1}, there exists a constant $D<1$ such that $(1- |z|^2) |\nabla u (z)| \leq 2 D u (z)$ for any $z \in \D$. Hence, there exists a universal constant $C>0$ such that for any arc $I \subset \partial \D$, we have
$$
\left|\log u(w)-\log u\left(z_{I}\right)\right| \leq C+ D \log \frac{|I|}{1- |w|}, \quad  w \in Q(I), 
$$
that is,
$$
\left(\frac{|I|}{1- |w|}\right)^{-D} \lesssim \frac{u(w)}{u\left(z_{I}\right)} \lesssim \left(\frac{|I|}{1- |w|}\right)^{D}, \quad  w \in Q(I). 
$$
Since $D<1$, we deduce
$$
\frac{1}{|I|^{2}} \int_{Q(I)} \frac{d A(w)}{u(w)} \lesssim \frac{1}{(1- D)u\left(z_{I}\right)}
$$
and Lemma \ref{meanvalue} shows that $u$ is a $B_{2}$ weight. 

\medskip

$(c) \Rightarrow(b)$. We argue by contradiction. Assume that for $N=1,2, \ldots$, there exists an arc $I = I_N \subset \partial \mathbb{D}$ such that
$$
\frac{\sigma (I)}{|I|} \leq 2^{-N} u\left(z_{I}\right) .
$$
Let $ J \subset 2^{-1} I$ be an arc with $|J|=2^{-n} |I|$, $1 \leq n \leq N$. Then

$$
\int_{\partial \mathbb{D} \setminus  I} \frac{1-|z_{J}|^{2}}{\left|\xi-z_{J}\right|^{2}} d \sigma (\xi)  \lesssim \frac{1-\left|z_{J}\right|^2}{1-\left|z_{I}\right|^2} \int_{\partial \mathbb{D} \setminus  I} \frac{1-|z_{I}|^{2}}{\left|\xi-z_{I}\right|^{2}} d \sigma (\xi) 
 \leq 2^{-n + 1} u\left(z_{I}\right) . 
$$
On the other hand
$$
\int_{I} \frac{1-\left|z_{J}\right|^{2}}{\left|\xi-z_{J}\right|^{2}} d \sigma(\xi) \leq \frac{ 2 \sigma(I)}{|J|} =2^{n+1} \frac{\sigma(I)}{|I|}
 \leq 2^{n+1 -N} u\left(z_{I}\right) .
$$
Thus, $u\left(z_{J}\right) \lesssim \left(2^{-n}+2^{n-N}\right) u\left(z_{I}\right)$, for any arc $J \subset 2^{-1} I$ with $|J|=2^{-n} |I| $, $1 \leq n \leq N$. Consider the set $E_n =  \{ z \in Q(I): 2^{-n-1} |I| \leq 1- |z| \leq 2^{-n} |I |\}$. Then 
$$
\frac{1}{|I|^2}  \int_{Q(I)} \frac{d A}{u} \geq \sum_{n=0}^N \frac{1}{|I|^{2}} \int_{E_n} \frac{d A}{u} \gtrsim \frac{1}{|I|^{2}} \frac{1}{u (z_{I} )} \sum_{n=0}^{N} \frac{2^{-n}|I|^{2}}{2^{-n}+2^{n-N}} \gtrsim  \frac{N}{u\left(z_{I}\right)} . 
$$
Since $N$ can be taken arbitrarily large, Lemma \ref{meanvalue} gives that $u$ is not a $ B_{2}$ weight and finishes the proof. 
\end{proof}

\vspace{0.2cm}
We now discuss the Remarks 3 and 4 mentioned in the introduction.

\textbf{Doubling measures.} As explained in Remark 2 from the introduction, condition (b) is equivalent to $\sigma (I) / |I| \simeq u(z_I)$, for any arc $I \subset \partial \D$. We now show that condition (b) implies that $\sigma$ is a doubling measure. Actually, condition (b) and  Harnack's inequality give 
$$
\frac{\sigma(I)}{|I|} \simeq u(z_{I}) \simeq u (z_{I'}) \simeq \frac{\sigma (I')}{|I'|},  
$$
for any pair $I, I'$ of adjacent arcs of $\partial \D$ of equal length. 

However, there are doubling measures for which (b) fails, as the following example shows. Let $\mathcal{D}_{n} $ be the collection of dyadic arcs of $\partial \D$ of generation $n$ and consider the measure $\sigma$ defined recursively by declaring its mass on any dyadic arc, as follows. Fix $0<p<1 / 2$. Define $\sigma (\partial \D)=1$ and by induction, assume that $\sigma(I)$ has been defined for any $I \in \mathcal{D}_{n}$. Fix an arc $I \in \mathcal{D}_n$ and write $I= I_{-} \cup I_+$, where $I_{-} \in  \mathcal{D}_{n+1}$ (respectively, $ I_+ \in \mathcal{D}_{n+1}$) is the left (respectively right) half of $I$. Define
\begin{align}
    \label{1exemple}
& \sigma \left(I_{-}\right)=p \sigma (I),\quad  \sigma\left(I_{+}\right)=(1-p) \sigma (I), \quad \text{if } n \text{ is odd}, \\
& \notag \sigma \left(I_{-}\right)=(1-p) \sigma (I),\quad  \sigma\left(I_{+}\right)=p \sigma (I), \quad \text{if } n \text{ is even}. 
 \end{align}
This defines the measure $\sigma$ and since $0 < p<1 / 2$, $\sigma$ is singular with respect to the Lebesgue measure in $\partial \D$. Note that 
\begin{equation}
    \label{4example}
    \frac{p}{1-p} \leq \frac{\sigma (I)}{\sigma \left(I'\right)} \leq \frac{1-p}{p}, 
\end{equation}
for any pair of contiguous arcs $I, I' \in \mathcal{D}_{n}$, $n \geq 1$. 
Since any arc in the unit circle contains a dyadic subarc  of comparable length, estimates \eqref{4example} imply that $\sigma$ is a doubling measure. Note that for any $n=1,2, \ldots$, there exists a dyadic arc $I_{n} \in \mathcal{D}_n$ such that $\sigma\left(I_{n}\right)=p^{n}$ while by Harnack's inequality, $u( z_{I_{n}} ) \geq 2^{-n}$. Hence, condition (b) implies that $p \geq 1/ 4$. In Section \ref{section5} we will show that $\sigma$ satisfies condition (b) if and only if $p > 1/4 $. 

\medskip

We will now show that if $\sigma$ is a doubling measure with doubling constant sufficiently close to $2$, then $\sigma$ satisfies condition (b). Assume that there exists $0< \varepsilon < 1$ such that $\sigma(2I) \leq 2 (1 + \varepsilon) \sigma (I)$ for any arc $I \subset \partial \D$. Then, standard estimates of the Poisson kernel and the doubling condition give
$$
u(z_I) =  \int_{\partial \mathbb{D}} \frac{1-\left|z_{I}\right|^{2}}{\left|\xi-z_{I}\right|^{2}} d \sigma(\xi) \lesssim \frac{\sigma (I)}{|I|} + \sum_{k=1}^\infty \frac{\sigma\left(2^{k} I \right)}{2^{2k}\left| I\right|} \lesssim \frac{2 \sigma (I) }{(1- \varepsilon) |I|}.  
$$
Hence, condition (b) holds. 

\textbf{$B_p$ weights.} Any positive harmonic function is a $B_{p}$ weight for any $p>2$. Let us sketch the proof. Let $u$ be a positive harmonic function in $\mathbb{D}$. By Harnack's inequality
$$
\frac{1-|w|}{1-|z_I|} \leq \frac{u(w)}{u(z_I)} \leq \frac{1-|z_I|}{1-|w|},  \quad w \in Q(I).
$$
In particular, $u(w) \geq 2^{-n} u(z_I)$ if $w \in Q(I)$ and $1-|w| \geq 2^{-n} |I|$. Set $E_n = \{ w \in Q(I): 2^{-n-1} |I| \leq 1-|w| \leq 2^{-n} |I| \}$. Then 
$$
\begin{aligned}
& \int_{Q(I)} u^{-\frac{1}{p-1}} d A \leqslant \sum_{n \geqslant 0} A(E_n) \left(2^{-n} u\left(z_{I}\right)\right)^{\frac{-1}{p-1}}  \simeq \\
& \simeq A(Q(I)) u\left(z_{I}\right)^{\frac{-1}{p-1}} \sum_{n \geqslant 0} 2^{\frac{n}{p-1}} 2^{-n} \simeq A(Q(I)) u\left(z_{I}\right)^{\frac{-1}{p-1}},  
\end{aligned}
$$
if $p>2$. Using Lemma \ref{meanvalue}, we deduce that $u$ is a $ B_{p}$ weight if $ p>2$.

On the other hand, there are positive harmonic functions $u$ in $\mathbb{D}$ such that $1 / u $ is not integrable in $\mathbb{D}$ and hence cannot be  $B_{2}$ weights. For example, one can take $u(z) = (1-|z|^2) |1 - z|^{-2}$, $ z \in \D$.

\section{Proof of Theorems \ref{Zeros} and \ref{mixing}}\label{section3}

Recall that for an analytic mapping $f: \D \rightarrow \D$ that factors as $f= E B  S_{\nu}$ where $E$ is an outer function, $B$ is a Blaschke product and $S_\nu$ is the singular inner function corresponding to the singular measure $\nu$, we associate the measure
\begin{equation}
\label{4radial}
\mu=\mu(f)=\sum_{z \in \D : f(z)=0}\left(1-|z|^{2}\right) \delta_{z}+2 \nu + 2 \log |f|^{-1} dm . 
\end{equation}
Denote by $P[\mu]$ the Poisson integral of $\mu$ on the closed unit disc defined as in the statement of Theorem \ref{Zeros}. We start with the following auxiliary result.

\begin{lem}
\label{lemma3}
Let $f: \D \rightarrow \D$ be an analytic mapping and let $\mu$ be the measure defined in \eqref{4radial}. Then

(a) $\log |f(z)|^{-2} \geqslant P[\mu](z)$ for any $ z \in \mathbb{D}$.

(b) For any $\delta>0$, there exists a constant $C=C(\delta)>0$ such that
\begin{equation*}
   \log |f(z)|^{-2} \leq C P [\mu ] (z) \quad \text { if } \quad d_{h}(z,\{w \in \mathbb{D} : f(w)=0 \}) \geqslant \delta.  
\end{equation*}
Moreover $C=C(\delta) \to 1$ as $\delta \to \infty$.

(c) If $f(z) \neq 0$, then 
$$
\frac{f'(z)}{f(z)}  = - \int_{\overline{\D}} \frac{d \mu (w)}{|1- \overline{w} z|^2 \tau_z (w)},  
$$
where $\tau_z (w) = (z-w)(1 - \overline{z} w)^{-1}$, $w \in \overline{\D}$, $z \in \D$.
\end{lem}

Proof. Let $\{z_n \}$ be the zeros of $f$ and let $\nu$ be the singular positive measure associated with the singular inner factor of $f$. Then
$$
\log |f(z)|^{-2}=\sum_{n} \log \left|\frac{z-z_{n}}{1-\overline{z_{n}} z}\right|^{-2}+2 P[\nu + \log |f|^{-1} dm](z), \quad z \in \mathbb{D} .
$$
The elementary estimate $-\log x \geqslant 1-x$, $0<x<1$, and the identity $1- |\tau_z (w)|^2 = (1-|z|^2)(1- |w|^2) |1 - \overline{z} w|^{-2}$ for $w \in \overline{\D}$, $z \in \D$, give (a). Part (b) follows similarly from the estimate $-\log x \lesssim 1-x$, $0<\delta <x<1$. Note that if $f=BSE$ is the inner-outer factorization of $f$, then
$$
\frac{f'(z)}{f(z)} = \frac{B'(z)}{B(z)} + \frac{S'(z)}{S(z)} + \frac{E'(z)}{E(z)}, 
$$
if $f(z) \neq 0$. Then (c) follows after computing these three logarithmic derivatives.
\qed

\begin{lem}
\label{primer}    
Let $f: \mathbb{D} \rightarrow \mathbb{D}$ be an analytic mapping. Assume that there exists a constant $0<c<1$ such that
\begin{equation}
    \label{1*}
\sup \left\{D_{h} f(z) :|f(z)| \geqslant c\right\}<1 .
\end{equation} 
Then $\sup \{D_{h} f(z): z \in \mathbb{D}\}<1$.
\end{lem} 

Proof. We argue by contradiction. Assume that there exists a sequence of points $z_{n} \in \mathbb{D}$ with $\left|f\left(z_{n}\right)\right| \leq c$ and $\left|D_{h} (f) \left(z_{n}\right)\right| \rightarrow 1$. 
Beardon and Minda proved (see Corollary 3.7 of \cite{BM}) that
$$
d_{h}\left(D_{h} (f) (z), D_{h} (f) (w)\right) \leq 2 d_{h}(z, w), \quad z, w \in \mathbb{D} .
$$
Then the assumption \eqref{1*} gives that
$$
k_{n}=d_{h}\left(z_{n},\{z \in \D:|f(z)| \geqslant c \} \right) \rightarrow \infty , \quad \text{ as } \, n \to \infty . 
$$
For $a \in \D$, let $\tau_a$ be the automorphism of $\D$ with $\tau_a^{-1} = \tau_a$ and $\tau_a (a) =0$. 
Consider $g_{n}= \tau_{f(z_n)} \circ f \circ\tau_{z_n}$, $n=1,2,\ldots$. Then $\left|g_{n}^{\prime}(0)\right|=D_{h} (f)(z_{n} )  \rightarrow 1$ as $n \to \infty$ and there exists a constant $0< C=C(c) < 1$ such that $|g_n (z)| \leq C$ if $d_{h}(z, 0) \leq k_{n} / 2$. Since $k_{n} \rightarrow \infty$ this contradicts Cauchy's estimates and finishes the proof.
\qed

\medskip

We are now ready to prove Theorem \ref{Zeros}. 

\begin{proof}[Proof of Theorem \ref{Zeros}]

We start showing that (b)$\implies$(a). As before, for $a \in \D$ let $\tau_a$ be the automorphism of $ \overline{\D}$ with $\tau_a^{-1} = \tau_a$ and $\tau_a (a) =0$, that is, $\tau_a (w) = (a-w)/ (1- \overline{a}w)$, $w \in \D$. Part (c) of Lemma \ref{lemma3} says 
\begin{equation}
\label{calculation}
\frac{(1- |z|^2) f^{\prime}(z)}{f(z)}= - \int_{\overline{\mathbb{D}}} \frac{\left(1-|z|^{2}\right) d \mu (w)}{|1-\bar{z} w|^{2} \tau_z (w)} , \quad z \in \mathbb{D} \setminus \{z \in \mathbb{D}: f(z) = 0 \} . 
\end{equation}

Let $C >0$ be the constant appearing in assumption (b). Let $C_0>0$ be a large constant to be fixed later. Part (a) of Lemma \ref{primer} says that $P[\mu ] (z) \leq \log |f(z)|^{-2}$, $ z \in \D$. Then, taking $c = c(C_o, C) <1$ sufficiently close to $1$, the Schwarz Lemma gives that $P[\mu] (z^*) \leq C$ if $d_h (z^* , \{z \in \D : |f(z)| >c \}) \leq C_0$. Fix $z \in \D$ with $|f(z)|> c$. Assumption (b) gives that for any Carleson box $Q \subset \D$ with $d_h (z_Q , z) \leq C_0$, we have 
$$
\frac{\mu(Q)}{l(Q)} \geq C P [\mu] (z_Q) .
$$
Fixing $C_0>0$ sufficiently large and arguing as in the implication (b) $\implies$ (a) of Theorem \ref{B2weights}, we show that there exists a constant $0<D<1$ independent of $z$ such that
$$
 \left|\int_{\overline{\mathbb{D}}} \frac{\left(1-|z|^{2}\right) d \mu (w)}{|1-\bar{z} w|^{2} \tau_z (w)} \right| 
 \leq D \int_{\overline{\mathbb{D}}} \frac{1-|z|^{2}}{|1-\bar{z} w|^{2}} d \mu(w)  . 
$$
Applying part (a) of Lemma \ref{lemma3} and identity \eqref{calculation}, we deduce that 
$$
(1-|z|^2) |f' (z)| \leq D |f(z)| \log |f(z)|^{-2}. 
$$
Using the elementary estimate $x \log (x^{-2})  \leq 1-x^{2}$, $0<x<1$, we deduce that $D_h (f) (z) \leq D$. Since this holds for any $z \in \D$ with $|f(z)|\geq c$, Lemma \ref{primer} implies that $f$ is contractive.

The converse implication (a)$\implies$(b) is proved by contradiction. If (b) does not hold, then for any $0 < \varepsilon < 1$ there exists a Carleson box $Q (\varepsilon)$ with
\begin{equation}
    \label{assum}
    \frac{\mu(\overline{Q(\varepsilon)})}{l(Q(\varepsilon))} \leq \varepsilon P[\mu]\left(z_{Q(\varepsilon)}\right) \text { and } P[\mu]\left(z_{Q(\varepsilon)}\right) \leq \varepsilon . 
\end{equation}
Then $ \mu (\overline{Q(\varepsilon)}) \leq \varepsilon^2 l (Q(\varepsilon))$. Let $z_\varepsilon \in \D$ be the point with $\operatorname{Arg} (z_\varepsilon ) = \operatorname{Arg} (z_{Q(\varepsilon)})$ and $1 - |z_\varepsilon| = \varepsilon^{1/4} l (Q(\varepsilon))$. Note that 
\begin{align}
P & [\mu] (z_\varepsilon)  = \int_{\overline{Q(\varepsilon)}} \frac{1-|z_\varepsilon|^2}{|1 - \overline{z_\varepsilon} w|^2} d\mu (w) + \int_{\overline{\mathbb{D}} \setminus \overline{Q(\varepsilon)}} \frac{1-|z_\varepsilon|^2}{|1 - \overline{z_\varepsilon} w|^2} d\mu (w) \leq \\\ 
 & \notag \leq \frac{2\mu(\overline{Q(\varepsilon)})}{1- |z_\varepsilon|} + \frac{2(1-|z_\varepsilon|^2)}{1-|z_{Q(\varepsilon)}|^2} P[\mu] (z_{Q(\varepsilon)}) \lesssim (\varepsilon^{7/4} + \varepsilon^{5/4} ) \leq C \varepsilon^{5/4} , 
\end{align}
where $C>0$ is an absolute constant. Since $P[\mu] (z_\varepsilon) \to 0$ as $\varepsilon \to 0$, we have that $d_h (z_\varepsilon, \{z \in \D : f(z)= 0 \})$ tends to $\infty$ as $\varepsilon $ tends to $0$. Then, parts (a) and (b) of Lemma \ref{lemma3} give  
\begin{equation}
    \label{zeros2}
  \lim_{\varepsilon \to 0}  \frac{ - \log |f(z_\varepsilon)|^2}{ P[\mu] (z_\varepsilon)} = 1 . 
\end{equation}
In particular, $|f(z_\varepsilon)| \to 1$ as $\varepsilon \to 0$. Now part (c) of Lemma \ref{lemma3} gives 
\begin{align}
    \label{1}
&    \left(1-\left|z_{\varepsilon}\right|^{2}\right) \frac{\left|f^{\prime}\left(z_{\varepsilon}\right)\right|}{\left|f\left(z_{\varepsilon}\right)\right|}=\left|\int_{\overline{\D} } \frac{1-\left|z_{\varepsilon}\right|^{2}}{\left|1-\bar{w} z_\varepsilon \right|^{2} \tau_{z_\varepsilon} (w)} d \mu(w)\right| \geq \\
& \geq \notag \left|\int_{\overline{\D}\setminus \overline{Q(\varepsilon)}} \frac{1-\left|z_{\varepsilon}\right|^{2}}{\left|1-\bar{w} z_\varepsilon \right|^{2} \tau_{z_\varepsilon} (w)} d \mu(w)\right| - C \int_{\overline{Q(\varepsilon)} } \frac{1-\left|z_{\varepsilon} \right|^{2}}{\left|1-\bar{w} z_\varepsilon \right|^{2}} d \mu(w),
\end{align}
where $C>0$ is an absolute constant. Note that
\begin{equation}
    \label{2}
    \int_{\overline{Q(\varepsilon)}} \frac{1-\left|z_{\varepsilon} \right|^{2}}{\left|1-\bar{w} z_\varepsilon \right|^{2}}  d \mu(w) \lesssim \frac{\mu (\overline{Q(\varepsilon)})}{\varepsilon^{1/4} l (Q(\varepsilon) )} \lesssim \varepsilon^{3/4} P[\mu] (z_{Q(\varepsilon)}) \leq C \varepsilon^{1/2} P[\mu] (z_{\varepsilon}),
\end{equation}
where in the second estimate we have used the assumption \eqref{assum} and in the last estimate we have applied Harnack's inequality. Since $\sup \{|\tau_{z_\varepsilon} (w) - z_\varepsilon| : w \in \overline{\D} \setminus Q(\varepsilon) \}$ tends to $0$ as $\varepsilon \to 0$, there exists $\eta(\varepsilon) \to 0$ as $\varepsilon \to 0$, such that 
\begin{equation*}
    \left|  \int_{\overline{\D} \backslash \overline{Q(\varepsilon)}} \frac{1-\left|z_{\varepsilon}\right|^{2}}{\left|1-\bar{w} z_{\varepsilon}\right|^{2} \tau_{z_\varepsilon} (w)} d \mu(w) \right| 
 \geqslant\left(1- \eta \left(\varepsilon \right)\right)   \int_{\overline{\D} \backslash \overline{Q(\varepsilon)}} \frac{1-\left|z_{\varepsilon}\right|^{2}}{\left|1-\bar{w} z_{\varepsilon}\right|^{2}} d \mu(w).  
\end{equation*}
Using \eqref{2}, we deduce 
\begin{align}
    \label{4}
& \left|  \int_{\overline{\D} \backslash \overline{Q(\varepsilon)}} \frac{1-\left|z_{\varepsilon}\right|^{2}}{\left|1-\bar{w} z_{\varepsilon}\right|^{2} \tau_{z_\varepsilon} (w)} d \mu(w) \right| 
 \geq  \\ \notag  & \geq (1- \eta (\varepsilon)) \bigg( P[\mu] (z_\varepsilon) - \int_{\overline{Q(\varepsilon)}} \frac{1-|z_\varepsilon|^2}{|1 - \overline{z_\varepsilon} w|^2} d \mu (w)\bigg)  
 \geq \\ \notag  & \geq (1 - \eta (\varepsilon)) (1 - C \varepsilon^{1/2}) P[\mu] (z_{\varepsilon}) .   
\end{align}
Now applying the estimates \eqref{2} and \eqref{4} in \eqref{1}, we obtain
\begin{equation}
    \label{5}
     \left(1-\left|z_{\varepsilon}\right|^{2}\right) \frac{\left|f^{\prime}\left(z_{\varepsilon}\right)\right|}{\left|f\left(z_{\varepsilon}\right)\right|} \geq \big( (1 - \eta (\varepsilon)) (1 - C \varepsilon^{1/2}) - C^2 \varepsilon^{1/2} \big) P [\mu] (z_\varepsilon). 
\end{equation}
Since $|f(z_\varepsilon)| \to 1$ as $\varepsilon \to 0$, identity \eqref{zeros2} and the elementary fact $\lim_{x \to 1} (1-x)^{-1} \log(1/x)  = 1$ show that $D_h (f) (z_\varepsilon) \to 1$ as $\varepsilon \to 0$. This finishes the proof.

\end{proof}

The proof of Theorem \ref{mixing} uses Theorem \ref{B2weights} and the Aleksandrov disintegration Theorem which says that the average of the Aleksandrov-Clark measures is the Lebesgue measure. More concretely, given an analytic mapping $f: \D \rightarrow \D$ consider the family $\{\sigma_\alpha : \alpha \in \partial \D\}$ of its Aleksandrov-Clark measures. Then for any integrable function $G$ on $\partial \D$ we have
\begin{equation}
    \label{ADT}
    \int_{\partial \D} G dm = \int_{\partial \D}  \bigg( \int_{\partial \D} G(\xi) d \sigma_\alpha (\xi) \bigg) dm (\alpha) . 
\end{equation}
See \cite{PolSar} or \cite{Saks}.

\begin{proof}[Proof of Theorem \ref{mixing}]
    
Let $\{ \sigma_{\alpha}: \alpha \in \partial \mathbb{D} \}$ be the family of Aleksandrov-Clark measures of $f$. The definition \eqref{AC11} of $\sigma_\alpha$ gives that
$$
\omega (f(z), E)= \int_{E} \frac{1-|f(z)|^{2}}{\left|\alpha -f(z)\right|^{2}} d m(\alpha)= 
 \int_{E} \int_{\partial \mathbb{D}} \frac{1-|z|^{2}}{|\xi-z|^{2}} d \sigma_{\alpha}(\xi) d m (\alpha) , 
$$
for any point $z \in \D$ and any measurable set $E \subset \partial \D $. Since $f$ is contractive, Theorem \ref{B2weights} provides a constant $C=C(f) >0$ such that 
\begin{equation}
    \label{0}
 C^{-1} \frac{\sigma_{\alpha}\left(I(z) \right)}{\left|I(z) \right|} \leq \int_{\partial \mathbb{D}} \frac{1-|z|^{2}}{|\xi -z|^{2}} d \sigma_{\alpha}(\xi) \leq C \frac{\sigma_{\alpha}\left(I(z) \right)}{\left|I(z) \right|}, \quad z \in \D, \alpha \in \partial \D . 
\end{equation}
Here $I(z)$ is the arc on $\partial \D$ centered at $z/|z|$ of length $1-|z|$. Hence
\begin{equation}
    \label{nou1}
    \omega(f(z), E) \simeq \int_{E} \frac{\sigma_{\alpha}\left(I(z) \right)}{\left|I(z) \right|} dm(\alpha), \quad z \in \mathbb{D} . 
\end{equation}
On the other hand, the Aleksandrov disintegration Theorem and the fact that $\sigma_{\alpha}$ is concentrated on $f^{-1} (\{\alpha \}) $ give
$$
 m\left(f^{-1}(E) \cap I \right)=\int_{\partial \D} \sigma_{\alpha}\left(f^{-1}(E) \cap I \right) d m(\alpha)
 =\int_{E} \sigma_{\alpha}(I) dm (\alpha) , 
$$
for any arc $I \subset \partial \D$. Applying \eqref{nou1} with $z=z_I$, we deduce that
$$
m\left(f^{-1}(E) \cap I\right) \simeq m(I) \omega \left(f(z_{I}), E \right) .
$$ 

We now prove the converse implication (b)$\implies$ (a). Given $a \in \mathbb{D}$ let $\tau_{a}=\tau_{a}^{-1}$ be the automorphism of $\D$ with $\tau(a)=0$. Fix $z \in \mathbb{D}$ and consider $g=\tau_{f(z)} \circ f \circ \tau_{z}$. Condition \eqref{0theoremmixing} gives that for any two arcs $I, J \subset \partial \mathbb{D}$ with $|I| \simeq |J| \simeq 1 $, we have $m\left(g^{-1}(J) \cap I\right) \simeq 1$. It follows that there exists a constant $c<1$, independent of $z$, such that
$$
\left|g^{\prime}(0)\right|=\left|\int_{\partial \mathbb{D}} \frac{g(\xi)}{\xi} dm (\xi)\right| \leqslant c<1 . 
$$
Since $D_h (f) (z) = |g'(0)|,$ this finishes the proof. \end{proof}

\begin{proof}[Proof of Corollary \ref{cor1}]
The claim follows from parts (b) of Theorems \ref{B2weights} and \ref{mixing}. 
\end{proof}    

\section{Proof of Theorems \ref{preimage} and \ref{limsup}}\label{section4}

\begin{proof}[Proof of Theorem \ref{preimage}]
Fix $\alpha \in \partial \D$ and let $\sigma=\sigma_{\alpha}$ be the Aleksandrov-Clark measure of $f$ at the value $\alpha$. Let $u$ be the Poisson integral of $\sigma$. Identity \eqref{tag1} says 
$$
\left(1-|z|^{2}\right) \frac{|\nabla u(z)|}{2 u(z)}=D_{h} f(z) \leq D, \quad z \in \mathbb{D} .
$$
Hence there exists a universal constant $C>0$ such that 
$$
\begin{aligned}
& |\log u(z)-\log u(0)| \leq \int_{0}^{z} \frac{|\nabla u(w)|}{u(w)}|d w| \leq \\
& \leqslant \int_{0}^{z} \frac{2 D|d w|}{1-|w|^{2}}= - D \log (1- |z|) + C , \quad z \in \mathbb{D} .
\end{aligned}
$$
Exponentiating, we get
$$
(1-|z|)^{D} \lesssim \frac{u(z)}{u(0)} \lesssim (1-|z|)^{-D},\quad  z \in \mathbb{D} . 
$$
Since $\sigma (I) / |I| \lesssim u\left(z_{I}\right)$ for any arc $I \subset \partial \D$, we deduce that $\sigma (I) \lesssim u(0)|I|^{1-D}$, for any arc $I \subset \partial \D$. Consequently $\sigma$ is absolutely continuous with respect to $(1-D)$-dimensional Hausdorff measure. Fix an arc $I \subset \partial \D$. Since $f$ does not extend analytically through $I$, we have $\sigma (I)>0$. Since $\sigma$ is concentrated on $ f^{-1}(\{\alpha\}) $, we deduce that $M^{1-D} (f^{-1}\{\alpha\} \cap I \ )>0$. 
\end{proof}

The same argument shows the following result. 

\begin{cor}
    \label{local}
    Let $f$ be an inner function such that
$$
\lim _{|z| \rightarrow 1} \frac{\left(1-|z|^{2}\right)\left|f^{\prime}(z)\right|}{1-|f(z)|^{2}}=0. 
$$
Then the Hausdorff dimension of $I \cap f^{-1}(\{\alpha\})$ is equal to $1$, for any arc $I \subset \partial \D$ and any $\alpha \in \partial \mathbb{D}$.
\end{cor} 

The proof of Theorem \ref{limsup} uses the following two auxiliary results. The first one is Lemma 3.1 in \cite{FN} but we include its short proof for the sake of completeness. 

\begin{lem}[\cite{FN}]
    \label{Lemma1.5}
    Let $0< \eta <c<1$ be constants and let $I \subset \partial \mathbb{D}$ be an arc. Let $\mathcal{G} = \{I_j \}$ be a collection of pairwise disjoint dyadic subarcs of $I$ such that
\begin{equation}
    \label{5.0}
    \sum |I_j| \geqslant c  |I| . 
\end{equation}
Then there exists a subcollection $\mathcal{G}_1 \subseteq \mathcal{G}$ such that 
\begin{equation}
\label{1radial}
\sum_{I_j \in \mathcal{G}_{1}}\left|I_{j}\right| \geqslant \eta|I| 
\end{equation}
and
\begin{equation}
    \label{2radial}
\sum_{I_j \in \mathcal{G}: I_{j} \subset L} |I_{j}|\geqslant(c-\eta) |L| , 
\end{equation}
for any dyadic subarc $L \subseteq I$ which contains some arc of $\mathcal{G}_1$.

\end{lem}

Proof. Consider the collection $\mathcal{F}$ of maximal dyadic subarcs $J \subseteq I$ such that
\begin{equation}
    \label{3radial}
\sum_{j : I_{j} \subset J }\left|I_{j}\right| \leq(c-\eta)|J|    
\end{equation}
and define $\mathcal{G}_{1}$ as the subcollection  of the arcs of $\mathcal{G}$ which are not contained in any arc of $\mathcal{F}$. Then \eqref{2radial} follows. Since the arcs in $\mathcal{F}$ are pairwise  disjoint, estimate \eqref{3radial} gives
$$
\begin{aligned}
& \sum_{\mathcal{J} \in \mathcal{G} \setminus \mathcal{G}_{1}}\left|J_{j}\right|=\sum_{J \in \mathcal{F}} \sum_{I_j \subset  J}\left|I_{j}\right| \leq \\
& \leq(c-\eta) \sum_{J \in \mathcal{F}}|J| \leq(c-\eta)|I| . 
\end{aligned}
$$
Then assumption \eqref{5.0} gives \eqref{1radial}.
\qed

\medskip

The next result is due to Hungerford and provides an estimate on the Hausdorff dimension of certain Cantor type sets. See \cite{Hun} or Theorem 10.5 of \cite{pom}. 

\begin{lem}[\cite{Hun}]
    \label{hungerford}
    Let $0<\varepsilon<c<1$ be fixed constants. Let $\mathcal{G}_{n}$ be nested collections of closed arcs of $\partial \mathbb{D}$ with pairwise disjoint interiors. Assume:
    
(a) We have $|J| \leq \varepsilon |I|$ for any pair of arcs $J \in \mathcal{G}_{n+1}$, $ I \in \mathcal{G}_{n}$ with $J \subset I$, for any $n=1,2, \ldots$.

(b) We have
$$
\sum_{J \in \mathcal{G}_{n+1}: J \subset I} |J| \geqslant c |I|, 
$$
for any $I \in \mathcal{G}_{n}$ and any $n=1,2, \ldots$.

Then 
$$
\operatorname{dim}\left(\bigcap_{n} \bigcup_{I \in \mathcal{G}_{n}} I \right) \geqslant 1-\frac{\log c}{\log \varepsilon}. 
$$
\end{lem}

\medskip

\begin{proof}[Proof of Theorem \ref{limsup}]
    
Let $\mu=\mu(f)$ be the measure defined in \eqref{4radial}. Since $f$ is inner the measure $\mu$ is a sum of Dirac masses at points in $\D$ and a measure on $\partial \D$ which is singular with respect to $dm$. The description of contractive analytic mappings given in Theorem \ref{Zeros} provides a constant $C>0$ such that
\begin{equation}
    \label{5radial}
\frac{\mu(Q)}{l(Q)} \geqslant C P[\mu ] \left(z_{Q}\right) \, \text{ if } \, P[\mu]\left(z_{Q}\right) \leqslant C . 
\end{equation}
Without loss of generality, we can assume $C<1$. Fix an arc $I \subset \partial \D$. Next we will find a constant $t=t(C) < 1$ and a Cantor type set contained in 
$$
E= \left\{\xi \in I: \limsup_{r \rightarrow 1} |f(r\xi)|<t \right\}
$$ 
and will estimate its Hausdorff dimension using Lemma \ref{hungerford}. 

Let $0<K_1 < K < 1/10$ be two constants satisfying $0<K_{1}< K / 10 <\min \left\{ C / 2, P [\mu] (z_{I}) \right\} / 100$. Since $f$ is inner the measure $\mu$ is singular with respect to the Lebesgue measure on $\partial \D$ and consequently
\begin{equation}
    \label{6radial}
    \lim_{r \rightarrow 1} P[\mu](r \xi)=0, \, \, \text{a.e.} \,  (d m) \,  \, \xi \in \partial \mathbb{D} . 
\end{equation}
Hence, there exists an arc $I^{(0)} \subset I$ such that $K_{1} / 2 \leq P[\mu] (z_{I^{(0)}}) \leq K_{1}$. Condition \eqref{5radial} gives $\mu\left(\overline{Q\left(I^{(0)} \right)}\right) \geq C K_{1} |I^{(0)}| / 2$. The arc $I^{(0)}$ is declared to be the generation zero of the Cantor type set. Let $\{I_j \}$ be the collection of maximal dyadic subarcs of $I^{(0)}$ such that
$$
P[\mu]\left(z_{I_{j}}\right) \geqslant K . 
$$
The maximality and Harnack's inequality give $P[\mu]\left(z_{I_{j}}\right) \leqslant 2 K$. Note that since $K<1 / 10$, the function $f$ does not have zeros in $Q\left(I^{(0)}\right) \backslash \bigcup Q\left(I_{j}\right)$. Moreover the restriction of $\mu$ to $\overline{I^{(0)}}$ is concentrated in $\cup \overline{I_j}$. Consequently, $\mu\left(\overline{Q\left(I^{(0)}\right)}\right)=\sum \mu\left(\overline{Q\left(I_{j}\right)}\right)$. Since $ 2K< C $, condition \eqref{5radial} gives 
\begin{equation}
\label{7radial}
\frac{\mu(\overline{Q(I_j)})}{|I_j|} \geqslant C K , \quad j=1,2,\ldots .
\end{equation}
Observe that 
$$
\frac{\mu\left(\overline{Q\left(I_{j}\right)}\right)}{\left|I_{j}\right|} \lesssim P [\mu] \left(z_{I_{j}}\right) \leqslant 2  K , \quad j=1,2,\ldots . 
$$
Hence there exists a universal constant $C_0 >0$ such that 
$$
\sum\left|I_{j}\right| \geqslant \frac{C_0}{  K} \sum \mu\left( \overline{Q\left(I_{j}\right)}\right)=\frac{C_0 \mu (\overline{Q (I^{(0)})})}{  K}  \geqslant \frac{C_0  C K_{1}}{2  K}\left|I^{(0)}\right| . 
$$
Apply Lemma \ref{Lemma1.5} with constants $O<\eta < C_0 C K_{1} (2  K)^{-1}:=C_{1}$ to find a subcollection $\mathcal{G}_{1}=\mathcal{G}_{1}\left(I^{(0)}\right)$ of $\{ I_{j}\} $ satisfying conditions \eqref{1radial} and \eqref{2radial} of Lemma \ref{Lemma1.5}. The collection $\mathcal{G}_1$ is the first generation of the Cantor type set. Condition \eqref{1radial} reads  
\begin{equation}
    \label{8radial}
    \sum_{I_j \in \mathcal{G}_1} |I_j| \geq \eta |I^{(0)}|.
\end{equation}
Let $L$ be a dyadic subarc of $I^{(0)}$ containing an arc of $\mathcal{G}_1$. By \eqref{2radial} of Lemma \ref{Lemma1.5}, we have
$$
\sum_{I_j \subset L } |I_j| \geq (C_1 - \eta ) |L|.
$$
Then \eqref{7radial} and Lemma \ref{lemma3} show that there exists a constant $0< c(K , K_1 , \eta ) < 1$ such that $|f(z_L)| \leq c(K, K_1, \eta) $. Hence
\begin{equation}
    \label{9radial}
    |f(r \xi)| \leq C(K, K_1 , \eta) < 1, \text{ if } r\xi \in Q(I^{(0)}) \setminus \bigcup_{I_j \in \mathcal{G}_1} Q(I_j) ,  \quad  \xi \in \bigcup_{\mathcal{G}_1} I_j . 
\end{equation}
Since $P [\mu] (z_{I^{(0)}})$ and $P[\mu] (z_{I_j})$, $I_j \in \mathcal{G}_1$, are small, Lemma 
\ref{lemma3} gives 
\begin{align}
\label{dues}
 1 & - |f(z_{I^{(0)}})|^2 \simeq \log  |f(z_{I^{(0)}})|^{-2} \simeq P [\mu] (z_{I^{(0)}}) \simeq K_1 , \\ 
\notag  1 & - |f(z_{I_j})|^2 \simeq \log  |f(z_{I_j})|^{-2} \simeq P [\mu] (z_{I_j}) \simeq  K  . 
\end{align}
Since $f$ is contractive, there exists a constant $D_1 < D$ such that 
$$ d_h (f(z_{I^{(0)}}), f (z_{I_j})) \leq D_1 d_h (z_{I^{(0)}}, z_{I_j}) .
$$ 
Using \eqref{dues} we deduce that there exists a constant $C_2 >0$ such that  
$$
\log \left(\frac{K}{K_1} \right) \leq D_1 \log \left (\frac{|I^{(0})}{|I_j|} \right) + C_2 , 
$$
that is, there exists a constant $C_3  >0$ such that 
\begin{equation}
    \label{10radial}
    \frac{|I_j|}{|I^{(0)}|} \leq C_3 \left(\frac{K_1}{K} \right)^{1/D_1} . 
\end{equation}
This finishes the construction of the first generation $\mathcal{G}_{1} =\mathcal{G}_{1}\left(I^{(0)}\right)$ of the Cantor type set. We now explain how the next generations are constructed. For each $I \in \mathcal{G}_{1}$ we have $K \leq P[\mu]\left(z_{I}\right) \leq 2 K$. Consider the family $\mathcal{F}(I)= \{ J_{k} \}$ of maximal dyadic arcs of $I$ such that
$$
P [\mu] \left(z_{J_{k}}\right) \leqslant K_{1} . 
$$
Note that \eqref{6radial} gives
\begin{equation}
    \label{11radial}
    \sum\left|J_{k}\right|=|I| . 
\end{equation}
By construction $P [\mu] (z) \geqslant K_1 / 2$, $z \in Q(I) \setminus \cup Q\left(J_{k}\right)$. 
By part (a) of Lemma \ref{lemma3}, we deduce
\begin{equation}
    \label{12radial}
    |f(z)| \leq c(K, K_1 , \eta) < 1, \quad z \in Q(I) \setminus \cup Q\left(J_{k}\right). 
\end{equation}
Now for each $J_{k} \in \mathcal{F}(I)$ we repeat the construction of the first generation replacing $I^{(0)}$ by $J_{k}$ and obtain a collection $\mathcal{G}_{1}\left(J_{k}\right)$ of subarcs of $J_{k}$. The second generation of the Cantor type set is then defined as
$$
\mathcal{G}_{2}=\bigcup_{I \in \mathcal{G}_{1}} \bigcup_{J_{k} \in \mathcal{F}(I)} \mathcal{G}_{1}\left(J_{k}\right) . 
$$
We continue the construction by induction. Estimates \eqref{9radial} and \eqref{12radial} give
\begin{equation}
\label{13radial}
    \limsup _{r \rightarrow 1}|f(r \xi)| \leq c\left(K, K_{1}, \eta \right)<1, \quad \xi\in \bigcap_{n} \bigcup_{I \in \mathcal{G}_{n}} I . 
\end{equation}
Estimates \eqref{8radial} and \eqref{11radial} give
$$
\sum_{J \in \mathcal{G}_{n+1}: J \subset I}|J| \geqslant \eta |I| \text {, for any } I \in \mathcal{G}_{n}, \quad n=1,2,\ldots , 
$$
and \eqref{10radial} says
$$
\frac{|J|}{|I|} \leq C_{3}\left(\frac{K_{1}}{K}\right)^{1 / D_{1}},
$$
for any $J \in \mathcal{G}_{n+1}, I \in \mathcal{G}_{n}$ with $J \subset I$, $n=1,2,\ldots$. Now Lemma \ref{hungerford} gives
$$
\operatorname{dim} \left(\bigcap_{n} \bigcup_{I \in \mathcal{G}_{n}} I \right) \geqslant 1-\frac{\log \eta}{\log C_{3}+ D_{1}^{-1} \log \left(\frac{K_{1}}{K}\right)} . 
$$
Note that we can take any $0<\eta <C_1= C_0 C K_{1} (2  K)^{-1}$. Fix $K>0$, pick $\eta= C_0 C K_1 (4  K)^{-1}$ and let $K_1 \to 0$ to deduce 
$$
\operatorname{dim} \left(\bigcap_{n} \bigcup_{I \in \mathcal{G}_{n}} I\right) \geqslant 1-D_1 , 
$$
and since $D_1 < D$, \eqref{13radial} completes the proof.

\end{proof}

Given an inner function $f$ and a hyperbolic disc $B \subset \D$, let $E(f, B)$ be the set of points $\xi \in \partial \D$ such that there exists $0< r_0 = r_0 (\xi) < 1$ with $f(r \xi ) \in B$ for any $r_0 < r < 1$. The conformal invariance of contractive analytic mappings gives the following version of Theorem \ref{limsup}.

\begin{cor}
Let $f$ be an inner function  with $\sup \{D_h (f) (z) : z \in \D\}<D<1$. 
Then there exists $M=M(D)>0$ such that $\operatorname{dim}  ( E(f, B) ) > 1- D$, for any hyperbolic disc $B$ of hyperbolic radius $M$. Moreorer, $M(D) \rightarrow 0$ as $D \rightarrow 0$.

\end{cor}

\section{Examples}\label{section5}

\subsection{Contractive domains}

As mentioned in the introduction, a domain $\Omega \subset  \mathbb{D}$ will be called a contractive domain if any analytic mapping $f: \D \rightarrow \Omega$ satisfies $\sup \{D_h (f) (z) : z \in \D\} < 1$. By a normal families argument, this is equivalent to the existence of a constant $D=D(\Omega)<1$ such that $\sup \{D_h (f) (z) : z \in \D\} <D$, for any analytic mapping $f: \mathbb{D} \rightarrow \Omega$. We will use the following elementary result.

\begin{lem}
    \label{2contractivedomain}
    Given $M>0$, there exists $0<\rho=\rho(M)<1$ such that if an analytic mapping $f: \D \rightarrow \mathbb{D}$ satisfies
$$
\sup _{z \in \mathbb{D}} \frac{\left(1-|z|^{2}\right)\left|f^{\prime}(z)\right|}{1-|f(z)|^{2}} \geqslant \rho , 
$$
then $f(\D)$ contains a hyperbolic disc of hyperbolic radius $M$.
\end{lem}

\begin{proof}
We argue by contradiction. Assume there exist $M>0$ and mappings $f_{n} : \mathbb{D} \rightarrow \mathbb{D}$, $n=1,2,\ldots$, with
$$
\sup _{z \in \mathbb{D}} \frac{\left(1-|z|^{2}\right) |f_{n}^{\prime} (z) |}{1-\left|f_{n}(z)\right|^{2}}>1-\frac{1}{n} , 
$$
such that $f_{n}(\D)$ does not contain any hyperbolic disc of hyperbolic radius $M$, $n=1,2,\ldots$. Replacing $f_{n}$ by $\tau_{f_n(a_n)} \circ f_{n} \circ \tau_{a_n} $, for appropriate $ a_{n} \in \D$, we can assume $\left|f_{n}^{\prime}(0)\right|>1- n^{-1}$. Then, taking a subsequence, we can assume that $\left\{f_{n}\right\}$ converges uniformly on compacts of $\D$ to a rotation. Hence $f_{n}(\D)$ contains hyperbolic discs of hyperbolic radius $M$ if $n$ is sufficiently large.
\end{proof}

\medskip

\begin{theorem}
    \label{contractivedomain}
    A domain $\Omega \subset \D$ is a contractive domain if and only if there exists $R=R(\Omega) >0$ such that $\Omega$ does not contain any hyperbolic disc of hyperbolic radius $R$.  
\end{theorem}

\begin{proof}[Proof of Theorem \ref{contractivedomain}]
    
For $a \in \D$ let $\tau_a =\tau^{-1}_a$ be the automorphism of $\D$ with $\tau_a (0)=a$. Assume that $\Omega$ contains a pseudohyperbolic disc $D$ centered at $a$ of pseudohyperbolic radius $r<1$. Then $f(z)=\tau_a (r z)$ satisfies $f(\D) \subset D \subset \Omega$ and $D_h (f) (0) = r$. Hence if $\Omega$ is a contractive domain, then $r \leq D(\Omega) < 1$. Consequently a contractive domain cannot contain hyperbolic discs of arbitrary large hyperbolic radius. The converse implication follows from Lemma \ref{2contractivedomain}.
\end{proof}

Let $K$ be a relatively compact subset of $\D$ and let $f$ be a universal covering map of $\D$ onto $ \D \setminus K$. Assume that $\sup \{d_h (z , K) : z \in \D \} < \infty$. Theorem \ref{contractivedomain} gives that $f$ is contractive. Note also that if $K$ has zero logarithmic capacity, then $f$ is inner. 

\subsection{Contractive versus bounded compression}

An analytic mapping $f : \D \rightarrow \D$ has bounded compression if the hyperbolic diameters of the images of hyperbolic discs of hyperbolic radius $1$ are uniformly bounded below. In \cite{ivriinicolau2025} several different descriptions of these mappings are presented. In this section we consider the singular doubling probability measure on $\sigma$ defined by \eqref{1exemple} 
and discuss when it leads to a contractive inner function or to an inner function of bounded compression.

We first show that $\sigma$ is an Aleksandrov-Clark measure of a contractive inner function  if and only if  $1 / 4<p<1 / 2$. By Theorem \ref{B2weights}, this amounts to showing that $P[\sigma]\left(z_{I}\right)$ is comparable to $ \sigma(I) / |I|$ for any arc $I \subset \D $ if and only if $1 / 4<p<1 / 2$.

Since $\sigma$ is a doubling measure, we can restrict our attention to dyadic arcs $I \subset \partial \D$. Fix a dyadic arc $I$ with $|I| =2^{-N}$. Standard estimates of the Poisson kernel show
$$
P [\sigma] (z_I ) \simeq \frac{\sigma(I)}{|I|}+\sum_{k = 1}^N \frac{\sigma\left(2^{k} I \setminus 2^{k-1} I\right)}{2^{2k}|I|} . 
$$
Let $I_{n} \in \mathcal{D}_n$ with $I \subset I_n$, $n=0, \cdots, N$. Since $\sigma$ is doubling, $\sigma (2^{k} I \setminus 2^{k-1} I) \simeq \sigma (I_{N-k})$. Then   
$$
P [ \sigma] (z_{I} ) \simeq \frac{\sigma(I)}{|I|}+\sum_{k=1}^{N} \frac{\sigma\left(I_{N-k}\right)}{2^{2 k}|I|}. 
$$
By construction $\sigma(I)=p^{\alpha(k, I)}(1-p)^{k-\alpha(k, I)} \sigma\left(I_{N-k}\right)$, where $\alpha(k, I)$ is the number of dyadic arcs $\left\{I_{n}: n=N-k+1, \ldots, N\right\}$ that have a proportion $p$ of the mass of their predecessor. Hence
$$
P [\sigma] (z_{I}) \simeq \frac{\sigma (I)}{|I|}+\sum_{k=1}^{N} \frac{p^{-\alpha(k, I)}{(1-p)}^{-(k-\alpha(k ,  I))}}{2^{2k}}  \frac{\sigma (I)}{|I|}. 
$$
We deduce that $P[\sigma]\left(z_{I}\right)$ is comparable to $ \sigma (I) / |I|$ if and only if
\begin{equation}
    \label{3example}
\sum_{k=1}^{N} \frac{p^{-\alpha(k, I)}(1-p)^{-k+\alpha(k, I)}}{2^{2 k}} \lesssim 1 .
\end{equation}
Since \eqref{3example} must be verified for every dyadic arc $I$ and $0<p<1/2$, the worst possible case occurs when $\alpha(k, I)=k$ for all $1 \leq k \leq N$, that is, when each predecessor of $I$ gets proportion $p$ of mass from its predecessor. Hence \eqref{3example} holds for all dyadic arcs $I$ if and only if
$$
\sum_{k=1} \frac{p^{-k}}{2^{2 k}} \lesssim 1,
$$
which holds if and only if $p>1 / 4$. 

\medskip

Finally, we show that for any $0<p<1 / 2$,  the measure $\sigma$ is the Aleksandrov-Clark measure of an inner function with bounded compression. We will use the following description of these measures given in \cite{ivriinicolau2025}: a positive singular measure $\sigma$ on $\partial \D$ is the Aleksandrov-Clark measure of an inner function of bounded compression if and only if for any $\varepsilon >0$ there exists $N=N(\varepsilon ) >0$ such that for any heavy arc $ I \subset \partial \D $ there exists a subarc $J \subset I$ with $|J| \geq 2^{-N} |I|$ such that $\sigma (J) / |J| \leq \varepsilon \sigma (I) / |I|$. Now observe that any dyadic arc $J \in \mathcal{D}_{n}$ contains a dyadic subarc $J \subset I$, $J \in \mathcal{D}_{n+k}$ with
$$
\frac{\sigma (J)}{|J|}=(2 p)^{k} \frac{\sigma (I)}{|I|}, 
$$
for any $k=1,2,\ldots$. Hence, given $\varepsilon >0$ we have 
$$
\frac{\sigma (J)}{|J|} <\varepsilon \frac{\sigma (I)}{|I|}
$$
if $k \geq (\log \varepsilon)(\log (2p))^{-1}$. Since $\sigma$ is a doubling measure, this shows that $\sigma$ is an Aleksandrov-Clark measure of an inner function of bounded compression.
\qed

\subsection{Composition operators}

Let $\mathcal{B}$ be the classical Bloch space of analytic functions $g: \D \rightarrow \mathbb{C}$ such that $\|g\| = \sup \{(1-|z|^2) |g' (z)| : z \in \D\} < \infty$. Given an analytic mapping $f: \D \rightarrow \D$, consider the composition operator $C_f : \mathcal{B} \rightarrow \mathcal{B}$ defined by $C_f (g) = g \circ f$, $g \in \mathcal{B}$. By the chain rule, 
\begin{equation}
    \label{comp}
    (1-|z|^2) |(g \circ f)' (z)|) = D_h (f) (z) (1-|f(z)|^2)|g' (f(z))| , \quad z \in \D, 
\end{equation}
and hence $\|C_f (g)\| \leq D(f) \|g \|$ for any $g \in \mathcal{B}$. Here $D(f) = \sup \{D_h (f) (z) : z \in \D \}$. It also shows that there exists a constant $D<1$ such that $\|C_f (g)\| \leq D \|g \|$ for any $g \in \mathcal{B}$ if and only if $f$ is contractive. It is well known that the essential norm $\|C_f\|_{e}$ satisfies 
$$
\|C_f\|_e = \lim_{c \to 1} \sup \{D_h (f) (z) : |f(z)|>c \}. 
$$
See \cite{Mon}. From Lemma \ref{primer} we deduce that $\|C_f \|_e < 1$ if and only if $f$ is contractive.

\vspace{-4pt}

\bibliographystyle{amsplain}

\vspace{2cm} 
\noindent
Departament de Matem\`atiques, 
\\Universitat Aut\`onoma de Barcelona and \\
Centre de Recerca Matem\`atica \\
08193 Barcelona\\
Email: \href{mailto:artur.nicolau@uab.cat}{\texttt{artur.nicolau@uab.cat}}

\end{document}